\documentclass[smallextended]{svjour3}
\usepackage{latexsym,amssymb,amsmath}
\smartqed
\usepackage{graphicx}
\RequirePackage{fix-cm}
\usepackage[section]{placeins}

\begin{document}
\title{\bf Differential Stability of Convex Discrete Optimal Control Problems}
\author {Duong Thi Viet An \and Nguyen Thi Toan}

\institute{Duong Thi Viet An \at Department of Mathematics and Informatics, College of Sciences,
		Thai Nguyen University, Thai Nguyen city, Vietnam\\
 \email{andtv@tnus.edu.vn}\\
 \\
 Nguyen Thi Toan \at School of Applied Mathematics and Informatics, Hanoi University of Science and Technology, 1 Dai Co Viet, Hanoi, Vietnam\\
   \email{toan.nguyenthi@hust.edu.vn}}
 
\date{Received: date / Accepted: date}

\titlerunning{Convex Discrete Optimal Control Problems}

\maketitle

\begin{abstract}

Differential stability of convex discrete optimal control problems in Banach spaces is studied in this paper. By using some recent results of An and Yen [Appl. Anal. 94, 108--128 (2015)] on differential stability of parametric convex optimization problems under inclusion constraints, we obtain an upper estimate for the subdifferential of the optimal value function of a parametric convex discrete optimal control problem, where the objective function may be nondifferentiable. If the objective function is differentiable, the obtained upper estimate becomes an equality. It is shown that the singular subdifferential of the just mentioned optimal value function always consists of the origin of the dual space.


\end{abstract}

\section{Introduction}
\label{intro}
Discrete  optimal control problems  (or  optimal control problems with discrete  time) arise when one has deal with controlled systems in which changes of the control and  current state can take place only at  strictly defined,  isolated instants of time.  

\medskip
Differential stability of parametric optimization problems is an important topic in variational analysis and optimization. In \cite{MordukhovichNamYen2009}, Mordukhovich, Nam and Yen gave
formulas for computing and estimating the Fr\'echet subdifferential, the Mordukhovich subdifferential, and the singular subdifferential of the optimal value function in parametric mathematical programming problems under inclusion constraints. If the problem in question is convex, by using the Moreau-Rockafellar
theorem and appropriate regularity conditions, An and Yao \cite{AnYao}, An and Yen \cite{AnYen} have obtained formulas for
computing subdifferentials of the optimal value function. In some sense, the results of \cite{AnYao} and \cite{AnYen} show that the preceding results of \cite{MordukhovichNamYen2009} admit a simpler form where several assumptions used in the general nonconvex case can be dropped.

\medskip
Besides the study on differential stability of  parametric mathematical programming problems, the study on differential stability of optimal control problems is also an issue of importance (see e.g. \cite{Chieu_Kien_Toan,Chieu_Yao,Kien_Liou_Wong_Yao}, \cite{Mouss1,TRk1,Seeger,Toan_Thuy,Toan_2015,Toan_Kien,Toan_Yao,Tu,BV1} and the references therein).

\medskip
Following the recent work of Chieu and Yao \cite{Chieu_Yao}, Toan and Yao \cite{Toan_Yao},
 this paper presents some new results on differential stability of convex discrete optimal control problems. Due to the convexity of the problem under our investigation, the results in \cite{AnYen} can be effectively used to yield an upper estimate for the subdifferential of the optimal value function of a parametric convex discrete optimal control problem, where the objective function may be nondifferentiable. If the objective function is differentiable, the obtained upper estimate becomes an equality. It is shown that the singular subdifferential of the just mentioned optimal value function always consists of the origin of the dual space. Our assumptions are weaker than those in
 \cite{Chieu_Yao} and \cite{Toan_Yao} applied to the convex case. In addition,  instead of the finite-dimensional spaces setting in those papers, here we can use a Banach spaces setting.

\medskip
The contents of the remaining sections are as follows. Section 2 formulates the control problem and recalls some auxiliary results from \cite{AnYen,Bonnans_Shapiro_2000,Ioffe_Tihomirov_1979,Mordukhovich_2006}. Differential stability of a specific mathematical programming problem is studied in Section~3 by invoking tools from functional analysis and infinite-dimensional convex analysis. Section~4 establishes three theorems on estimating/computing subdifferentials of the optimal value function of the parametric convex discrete control problem. The last section shows how these theorems can be used for analyzing concrete problems.

\section{Problem Formulation and Auxiliary Results} \markboth{\centerline{\it Preliminaries}} {\centerline{\it D. T. V.~An and N. T.~Toan
	}} \setcounter{equation}{0}

This section is divided into four subsections. The first one introduces the \textit{convex discrete optimal control problem} that we are interested in. The second one transforms the problem to a parametric convex optimization problem under an inclusion constraint. The third one recalls several basic concepts from variational analysis and the last one gives some facts from functional analysis and convex analysis, which are needed for studying the above convex discrete optimal control problem.

\subsection{Control problem}
 	
Let $X_k$, $U_k$, $W_k$, for $k=0, 1,\ldots, N-1$ and $X_N$, be Banach spaces, where $N$ is a positive natural number. Let there be given

- convex sets  $\Omega_0\subset U_0, \dots, \Omega_{N-1}\subset U_{N-1}$,  and $C\subset X_0$;

- continuous linear operators $A_k:X_k\to X_{k+1},\ B_k:U_k\to X_{k+1},\ T_k: W_k\to X_{k+1}$, for $k=0, 1,\ldots, N-1$;

- functions $h_k: X_k\times U_k \times W_k\to \mathbb{R}$ for $k=0, 1,\ldots, N-1$, and $h_N: X_N \to \mathbb{R}$, which are convex.

\medskip
We are going to describe a control system where the state variable (resp., the control variable) at time $k$ is $x_k$ (resp., $u_k$), and the objective function is the sum of the functions $h_k$, for $k=0, 1,\ldots, N$. We interpret $X_k$ as the space of state variables at stage $k$, and $U_k$ (resp., $W_k$) the space of control variables (resp., space of random parameters) at stage $k$.

\medskip
Put  $W= W_0\times W_1\times\cdots\times W_{N-1}$. For every vector 
 $w=(w_0, w_1, \ldots, w_{N-1})\in W$, consider the following \textit{convex discrete optimal control problem}: Find a pair $(x,u)$ where $x=(x_0, x_1,\ldots,x_N)\in
 X_0\times X_1\times \cdots\times X_N$ is a trajectory and $u=(u_0,
 u_1,\ldots, u_{N-1})\in U_0\times U_1\times \cdots\times U_{N-1}$ is a control sequence, which minimizes the \textit{objective function}
 \begin{eqnarray}\label{problem1}
 \sum_{k=0} ^{N-1} h_k(x_k, u_k, w_k)+ h_N(x_N)
 \end{eqnarray} and satisfies the \textit{linear state equations}
 \begin{eqnarray}x_{k+1}= A_k x_k +B_k u_k +  T_k w_k,\ k=0, 1,\ldots, N-1,
 \end{eqnarray} 
 the \textit{initial condition}
 \begin{eqnarray}
 x_0\in  C,
 \end{eqnarray}  
 and the \textit{control constraints}
 \begin{eqnarray}\label{problem4}
 u_k\in \Omega_k\subset U_k, \ \; k=0,1,\ldots, N-1.
 \end{eqnarray}
 
 A classical example for the problem \eqref{problem1}--\eqref{problem4} is the \textit{inventory control problem} in economics, where $x_k$ plays a stock available at the beginning of the $k$th period, $u_k$ plays a stock ordered (and immediately delivered) at the beginning of the $k$th period, $w_k$ is the demand during the $k$th period (in practice, $w_0,\dots,w_{N-1}$ are independent random variables with a given probability distribution), and the objective function has the form $h(x_N)+\displaystyle\sum_{k=0}^{N-1}h(x_k,u_k,w_k)$
   together with the state equation $x_{k+1}=x_k+u_k-w_k$ (see \cite[pp. 2--6, 13--14, 162--168]{Bertsekas} for details).
 
  \subsection{Reduction to a parametric optimization problem}
  
   Put  $X= X_0\times X_1\times\cdots\times X_N,\,U= U_0\times
   U_1\times\cdots\times U_{N-1}$. For every parameter $w=(w_0, w_1,\ldots, w_{N-1})\in W$, denote by $V(
   w)$ the optimal value of problem \eqref{problem1}--\eqref{problem4}, and by $S(w)$ the solution set of that problem. Thus, $V: W\to
   \mathbb{\overline R}$ is an extended real-valued function which is
   called \textit{the optimal value function} of problem \eqref{problem1}--\eqref{problem4}. It is
   assumed that $V$ is finite at a certain parameter $\bar w=(\bar w_0, \bar
   w_1,\ldots,\bar w_{N-1})\in W$ and $(\bar x, \bar u)$ is a
   solution of \eqref{problem1}--\eqref{problem4}, that
   is $(\bar x, \bar u)\in
   S(\bar w)$ where $\bar x=(\bar x_0, \bar x_1,\ldots, \bar x_N), \bar u=(\bar u_0, \bar u_1,\ldots, \bar
   u_{N-1})$.

\medskip   
   For each $w=(w_0, w_1,\ldots, w_{N-1})\in W$, let
   \begin{eqnarray}\nonumber f(x, u, w)=
   \sum_{k=0} ^{N-1} h_k(x_k, u_k, w_k)+h_N(x_N).
   \end{eqnarray}
   Then, setting $ \Omega=
   \Omega_0\times\Omega_1\times\cdots\times\Omega_{N-1},\
   \widetilde{X}=X_1\times X_2\times\cdots\times X_N$, and  $$G(w)=\{(x,u)\in X \times U\, | \, x_{k+1}= A_k x_k +B_k u_k +  T_k w_k,\ k=0, 1,\ldots, N-1\},$$ we have $$V(w)=\inf_{(x,u)\in G(w)\cap (C\times
   	\widetilde{X}\times\Omega)} f(x,u,w).$$ 
     
  \subsection{Three dual constructions}
   
  We will need three dual constructions: {\it normal cone} to convex sets, \textit{subdifferential and singular subdifferential}  of convex functions, and \textit{coderivative} of convex multifunctions. Let $X$ and $Y$ be Hausdorff locally convex topological vector spaces with the topological duals denoted, respectively, by $X^*$ and $Y^*$. For a convex set $\Omega\subset X$, the {\it normal cone} of $\Omega$ at $\bar x\in \Omega$ is given by \begin{align}\label{normals_convex_analysis} N(\bar x; \Omega)=\{x^*\in X^* \mid \langle x^*, x-\bar x \rangle \leq 0, \ \, \forall x \in \Omega\}.\end{align} 
  
  Let $f: X\rightarrow \overline{\mathbb{R}}$, where $\overline{\mathbb{R}}=[- \infty, + \infty]$, be an extended real-valued function.  One says that $f$ is \textit{proper} if the \textit{domain}
  ${\rm{dom}}\, f:=\{ x \in X \mid f(x) < +\infty\} $
  is nonempty, and if $f(x) > - \infty$ for all $x \in X$. The \textit{epigraph} of $f$ is the set $ {\rm{epi}}\, f:=\{ (x, \alpha) \in X \times \mathbb{R} \mid \alpha \ge f(x)\}$. If the latter set is convex, then $f$ is said to be a \textit{convex function.}
  
  \medskip
  The {\it subdifferential} of a proper convex function $f: X\rightarrow \overline{\mathbb{R}}$ at a point $\bar x \in {\rm dom}\,f$ is defined by
  \begin{align}\label{subdifferential_convex_analysis}
  \partial f(\bar x)=\{x^* \in X^* \mid \langle x^*, x- \bar x \rangle \le f(x)-f(\bar x), \ \forall x \in X\}. 
  \end{align} Note that $x^* \in \partial f(\bar x)$ if and only if $\langle x^*, x-\bar x \rangle -(\alpha- f(\bar x)) \le 0$ for all $(x, \alpha) \in {\rm epi}\,f$
  or, equivalently, $(x^*,-1)\in N( (\bar x, f(\bar x)); {\rm epi}\, f)$. Therefore, \begin{align*}
  \partial f(\bar x)=\{x^* \in X^* \mid (x^*,-1)\in N ( (\bar x, f(\bar x)); {\rm{epi}}\, f)\}.
  \end{align*} The {\it singular subdifferential} of a convex function $f$ at a point $\bar x \in {\rm dom}\,f$ is given by
  \begin{align}\label{singular_subdifferential_convex}
  \partial^\infty f(\bar x)=\{x^* \in X^* \mid (x^*,0)\in N ( (\bar x, f(\bar x)); {\rm{epi}}\, f)\}.
  \end{align} 
  For any  $\bar x\notin {\rm dom}\,f$, one puts $\partial f(\bar x)=\emptyset$ and $\partial^\infty f(\bar x)=\emptyset$. It is easy to see that $\partial \delta (x; \Omega)=N(x;\Omega)$ where $\delta (\cdot;\Omega)$ is the {\it indicator function} of a convex set $\Omega \subset X$. Recall that $\delta (x; \Omega)=0$ if $x \in \Omega$ and $\delta (x; \Omega)=+\infty$ if $x \notin \Omega$. Interestingly, for any convex function $f$, one has $\partial^\infty f(x)=N(x;{\rm dom}\,f)$; see e.g. \cite{AnYen}.
  
  \medskip
  One says that a multifunction $F:X\rightrightarrows Y$ is {\it closed} (resp., {\it convex}) if ${\rm gph}\,F$ is closed (resp., convex). The \textit{coderivative} $ D^* F(\bar x, \bar y): Y^* \rightrightarrows X^*$  of a convex multifunction $F$ between $X$ and $Y$ at $(\bar x, \bar y) \in {\rm{gph}}\, F$ is the multifunction defined by
  \begin{align}\label{coderivative_convex_case}
  D^* F(\bar x, \bar y)(y^*):=\left\{x^* \in X^* \mid (x^*, -y^*) \in  N ( (\bar x, \bar y); {\rm{gph}}\, F)\right\}, \ \forall y^* \in Y^*.
  \end{align}  
  If $(\bar x, \bar y) \notin {\rm{gph}}\, F$, one puts $D^* F(\bar x, \bar y)(y^*)=\emptyset$ for any $y^* \in Y^*$.
  
  \subsection{Some facts from functional analysis and convex analysis}
 
  First, we recall a result related to continuous linear operators. Let  $A:X\rightarrow Y$ be a continuous linear operator from a Banach space $X$ to another Banach space $Y$ with the  adjoint $A^*:Y^* \rightarrow X^*$. The null space and the range of $A$ are defined respectively by
  ${\rm ker}\,A=\{x\in X\, \mid \, Ax=0 \}$ and 
  $${\rm rge}\,A =\{y \in Y\, \mid \, y=Ax, \, x\in X \}.$$
  \begin{proposition} \label{proposition2.173}{\rm{(See \cite[Proposition~2.173]{Bonnans_Shapiro_2000})}}
  The following properties are valid:\\
  {\rm (i)} $(\rm ker\,A)^\bot={\rm cl}^*({\rm rge}\,(A^*))$ with
  $({\rm ker}\,A)^\bot=\{ x^*\in X^*\, \mid \, \langle x^*, x\rangle =0\ \, \forall x \in{\rm ker}\,A \}$ and ${\rm cl}^*({\rm rge}\,(A^*))$ denoting  respectively the orthogonal complement of the set $ {\rm ker}\,A$ and the closure of the set ${\rm rge}\,(A^*)$ in the weak$^*$ topology of $X^*$.\\
  {\rm (ii)} If ${\rm rge}\,A$ is closed, then $({\rm ker}\,A)^\bot={\rm rge}\,(A^*),$ and there is $c>0$ such that for every $x^*\in {\rm rge}\,(A^*)$ there exists $y^*\in Y^*$ with $||y^*|| \le c||x^*||$ and $x^*=A^*y^*$.\\
   {\rm (iii)} If, in addition, ${\rm rge}\,A=Y$, i.e., $A$ is onto, then $A^*$ is one-to-one and there exists $c>0$ such that $||y^*||\le c||A^*y^*||$, for all $y^*\in Y^*$.\\
   {\rm (iv)} $({\rm ker}\,A^*)^\bot={\rm cl( rge}\,A).$
  \end{proposition}
  
   Next, we recall two results on normal cones to convex sets. Let $A_0, A_1,\dots, A_n$ be convex subsets of a Banach space $X$ and let $A=A_0\cap A_1\cap\dots\cap A_n$. By ${\rm{int} }\, A_i$, for $i=1,\dots,n$, we denote the interior of $A_i$  in the norm topology of $X$.
  
  \begin{proposition}\label{proposition1}{\rm{(See \cite[Proposition~1, p.~205]{Ioffe_Tihomirov_1979})}} If
  	$A_0 \cap ({\rm{int} }\, A_1)\cap\dots\cap ({\rm{int}}\, A_n)\not=\emptyset$,  then
  	$$N(x;A)=N(x;A_0)+N(x;A_1)+\dots+N(x;A_n)$$
  	for any point $x\in A$. In the other words, the normal cone to the intersection of sets is equal to the sum of the normal cones to these sets. 	
  \end{proposition}
  	\begin{proposition}\label{proposition3}{\rm{(See \cite[Proposition~3, p.~206]{Ioffe_Tihomirov_1979})}}
  		If ${\rm int}\,A_i\not = \emptyset$ for $i=1,2,\dots,n$ then, for any  $x_0\in A$, the following statements are equivalent:\\
  		  {\rm (a)} $A_0 \cap ({\rm{int} }\, A_1)\cap\dots\cap ({\rm{int}}\, A_n)=\emptyset$;\\
  		  {\rm (b)} There exist $x_i^*\in N(x_0;A_i)$ for $i=0,1,\dots,n$, not all zero, such that $$x_0^*+x_1^*+\dots+x_n^*=0.$$
  	\end{proposition}
  	\section{Differential stability of the parametric mathematical programming problem}
  	 By using some recent results from \cite{AnYen} on differential stability of parametric convex optimization problems under inclusion constraints, this section establishes a theorem, which is the main tool for our subsequent investigations on the discrete optimal control problem. 	
  	 	\medskip
  	 	
  	 	 Let $\varphi: X \times Y \rightarrow \overline{\Bbb R}$ be a proper convex function, $G: X \rightrightarrows Y$ a convex multifunction between Hausdorff locally convex topological vector spaces. Consider the parametric optimization problem under an inclusion constraint
  	 	 \begin{align}\label{math_program}
  	 	 	\min\{\varphi(x,y)\mid y \in G(x)\}
  	 	 \end{align}
  	 	 depending on the parameter $x$, 
 	 	 with the optimal value function
 	 	   	 	 $\mu: X \rightarrow \overline{\mathbb{R}}$
  	 	 defined by
  	 	 \begin{align}\label{marginalfunction}
  	 	   	 	 	\mu(x):= \inf \left\{\varphi (x,y)\mid y \in G(x)\right\}.
  	 	   	 	 \end{align}
  	 	   	 	  The solution map $M: {\rm {dom}}\, G  \rightrightarrows Y $ of problem \eqref{math_program} is
  	 	 \begin{align}\label{solution_map}
  	 	   	 	 	M(x):=\{y \in G(x)\mid \mu(x)= \varphi (x,y)\}.
  	 	   	 	 \end{align}
  	 	 
  	 	 The problem of computing the subdifferential and singular subdifferential of $\mu(\cdot{)}$ has been considered in \cite{AnYen} (the Hausdorff locally convex topological vector spaces setting) and in \cite{AnYao} (the Banach space setting). The following result of \cite{AnYen} will be used intensively in this paper. 
  	 	 
  	 	\begin{theorem}\label{theoremAnYen}{\rm{(See \cite[Theorem 4.2]{AnYen})}}
  	 		If at least one of the following regularity conditions is satisfied\\
  	 		{\rm (i)} ${\rm{int}}({\rm gph}\,G)\cap {\rm{dom}} \, \varphi \ne\emptyset$,\\
  	 		{\rm (ii)} $\varphi$ is continuous at a point $(x^0, y^0)\in {\rm gph}\,G$,\\
  	 		then, for any $\bar x \in {\rm{dom}}\, \mu$, with $\mu(\bar x) \not= - \infty$, and for any $\bar y \in M(\bar x)$ one has
  	 		\begin{align*}
  	 			\partial \mu(\bar x) = \bigcup\limits_{(x^*,y^*) \in \partial \varphi(\bar x, \bar y)}   \big\{x^* + D^*G( \bar x, \bar y)(y^*) \big\}
  	 		\end{align*}
  	 		and
  	 		\begin{align*}
  	 			\partial^\infty \mu(\bar x) = \bigcup\limits_{(x^*,y^*) \in \partial^\infty \varphi(\bar x, \bar y)}   \big\{x^* + D^*G( \bar x, \bar y)(y^*) \big\}.
  	 		\end{align*}
  	 	\end{theorem}
  	 	\medspace
  	 	
  	 	We now specify Theorem \ref{theoremAnYen} for a case where ${\rm gph}\,G$ is a linear subspace of a product space. Suppose that $X$, $W$ and $Z$ are
  	 	Banach spaces with the dual spaces $X^*$, $W^*$ and $Z^*$, respectively. Assume that $M: Z\to X$ and $T: W\to X$
  	 	are continuous linear operators. Let $M^*: X^*\to Z^*$ and $T^*: X^*\to W^*$
  	 	be the adjoint operators of $M$ and $T$, respectively. Let 
  	 	$f:W\times Z\to\overline{\Bbb{R}}$ be a convex extended real-valued function and $\Omega$ a convex subset of $Z$ with
  	 	nonempty interior. For each $w\in W$, put
  	 	$ H(w)=\big\{z\in Z\,|\, Mz=Tw\big\}$ and consider the optimization problem 
  	 	\begin{eqnarray} \label{12b}
  	 	\min\{f(z,w) \, \mid \,z\in H(w)\cap \Omega \}.
  	 	\end{eqnarray} We want to compute the subdifferential and the singular subdifferential of the optimal value function
  	 	\begin{eqnarray} \label{12}
  	 	h(w):=\inf_{z\in H(w)\cap \Omega} f(z,w)
  	 	\end{eqnarray} of the parametric problem \eqref{12b}. Denote by $\widehat S(w)$ the solution set of \eqref{12b}. 
  	 	
  	 	\medskip
  	 Define the linear operator $\Phi: W \times Z \rightarrow X$ by setting $\Phi(w,z)=-Tw+Mz$ for all $(w,z)\in W \times Z$. 
  	  	
  	 	\begin{lemma}\label{bd1}   For each  $(\bar w, \bar z)\in{\rm gph}\,H$,
  	 		one has
  	 		\begin{eqnarray} \label{normalcone1}
  	 		N\big((\bar w, \bar z);{\rm gph}\,H \big)= {\rm cl}^*\big\{(-T^*x^*, M^*x^*)\,|\,
  	 		x^*\in X^*\big\}.
  	 		\end{eqnarray}
  	 		Moreover, if $\Phi$ has closed range, then
  	 		\begin{eqnarray}\label{normalcone2}
  	 		N\big((\bar w, \bar z);{\rm gph}\,H \big)= \big\{(-T^*x^*, M^*x^*)\,|\,
  	 		x^*\in X^*\big\}.
  	 		\end{eqnarray}
  	 		In particular, if $\Phi$ is surjective, then \eqref{normalcone2} is valid.
  	 	\end{lemma}
  	 	\begin{proof} First, note that $\Phi$ is continuous by the continuity of $T$ and $M$. Second, observe that 
  	 		$${\rm gph}\,H  =\{(w,z) \mid \Phi(w,z)=0 \}= \Phi^{-1}(0)={\rm ker}\, \Phi.$$ 
  	 		On one hand, we have \begin{align}\label{range_of_Phi-star}\Phi^*(x^*)=(-T^*x^*,M^*x^*)\ \; \forall x^*\in X^*,\end{align} because
  	 		\begin{align*}
  	 		\langle \Phi^*(x^*), (w,z)\rangle &= \langle x^*, \Phi(w,z) \rangle\\
  	 		&	= \langle x^*, -Tw \rangle + \langle x^*, Mz \rangle\\ & = \langle -T^*x^*, w \rangle + \langle M^*x^*, z \rangle \\
  	 		&	= \langle (-T^*x^*, M^*x^*),(w,z)\rangle
  	 		\end{align*}
  	 		for every $(w,z)\in W \times Z$. On the other hand, since ${\rm gph}\,H $ is a linear subspace of $W\times Z$, 
  	 		\begin{align}
  	 		\label{formula_proof2}
  	 		N((\bar w, \bar z); {\rm gph}\,H )=({\rm gph}\,H )^\bot=({\rm ker}\, \Phi )^\bot,
  	 		\end{align}
  	 		where  $$({\rm ker}\, \Phi )^\bot=\{(w^*,z^*)\in W^*\times Z^*\mid \langle (w^*,z^*), (w,z) \rangle =0\ \, \forall (w,z) \in {\rm ker}\, \Phi \}.$$ 
  	 		Hence, by the first assertion of Proposition \ref{proposition2.173}, \eqref{normalcone1} follows from \eqref{range_of_Phi-star} and \eqref{formula_proof2}. If $\Phi$ has closed range, then the weak$^*$ closure sign in \eqref{normalcone1} can be removed due to the second assertion of Proposition \ref{proposition2.173}. Thus,  \eqref{normalcone2} is valid.
  	 	 		If $\Phi$ is a surjective, then it has closed range; so  \eqref{normalcone2} holds true. 
  	 	 		$\hfill\Box$
  	 	\end{proof}
  	 	
  	 	\begin{lemma}\label{bd2}  If $\Phi$ has closed range and ${\rm ker}\, T^* \subset {\rm ker}\,M^*$ then, for each  $(\bar w, \bar z)\in{\rm gph}\,H$, one has
  	 		\begin{equation}\label{decomp}N\big((\bar w, \bar z); (W \times \Omega)\cap {\rm gph}\,H  \big)= \{0\}\times N(\bar z; \Omega) + N\big((\bar
  	 		w, \bar z); {\rm gph}\,H \big).\end{equation}
  	 	\end{lemma}
  	 	\begin{proof} First, let us show that 
  	 		\begin{equation}\label{reg1new}
  	 		N ((\bar w, \bar z);W \times \Omega )\cap [ - N((\bar w, \bar z); {\rm gph}\,H  )]=\{(0,0)\}. 
  	 		\end{equation}
  	 		To obtain this property, take any $$(w^*,z^*)\in N ((\bar w, \bar z);W \times \Omega )\cap [ - N((\bar w, \bar z); {\rm gph}\,H  )].$$ Since $N ((\bar w, \bar z);W \times \Omega )=\{0\}\times N(\bar z; \Omega),$ we must have $w^*=0,\, z^*\in N(\bar z; \Omega).$ 
  	 		As $\Phi$ has closed range, \eqref{normalcone2} is valid by Lemma \ref{bd1}. Therefore, the inclusion  $(w^*,z^*)\in  - N((\bar w, \bar z); {\rm gph}\,H  )$ implies the existence of $x^*\in X^*$ such that
  	 		$0=T^*x^*$ and $z^*=-M^*x^*$. Combining this with the inclusion  ${\rm ker}\, T^* \subset {\rm ker}\,M^*$, we obtain $z^*=0$. The property \eqref{reg1new} has been proved.
  	 		
  	 		Next, since $ {\rm int} \, \Omega \not= \emptyset$, we see that $W \times \Omega$ is a convex set with nonempty interior. Let $A_0:={\rm gph}\,H $ and $A_1:=W \times \Omega$. Due to \eqref{reg1new}, one cannot find any $(w^*_0,z^*_0)\in N((\bar w, \bar z);A_0)$ and $(w^*_1,z^*_1)\in N((\bar w, \bar z);A_1)$, not all zero, such that $(w^*_0,z^*_0)+(w^*_1,z^*_1)=0$. Hence, applying Proposition \ref{proposition3} to the sets $A_0$ and $A_1$ and the point $(\bar w, \bar z)\in A_0\cap A_1$, we can assert that  $A_0\cap {\rm int}\,A_1\neq\emptyset$. Therefore, by Proposition \ref{proposition1} we have \begin{equation}\label{decomp1}N((\bar w, \bar z);A_0\cap A_1)=N((\bar w, \bar z);A_0)+N((\bar w, \bar z);A_1).\end{equation} Since $N((\bar w, \bar z);A_0)=N\big((\bar
  	 		w, \bar z); {\rm gph}\,H \big)$ and $N((\bar w, \bar z);A_1)=\{0\}\times N(\bar z; \Omega)$, the equality \eqref{decomp} follows from \eqref{decomp1}.
  	 			$\hfill\Box$
  	 	\end{proof}
  	 
  	  	\begin{theorem}\label{asprogramingproblem} 
  	 		Suppose that $\Phi$ has closed range and ${\rm ker}\, T^* \subset {\rm ker}\,M^*$. If the optimal value function $h$ in \eqref{12} is finite at $\bar w \in {\rm dom}\, \widehat{S}$ and $f$ is continuous at $( \bar w, \bar z) \in  (W \times \Omega) \cap {\rm gph}\,H ,$ then
  	 		\begin{eqnarray}\label{BT1}
  	 		\partial h(\bar w)= \bigcup_{(w^*, z^*)\in \partial f(\bar z, \bar w)}\;\bigcup_{v^*\in N(\bar z;
  	 			\Omega)}\big[ w^*+ T^*\big((M^*) ^{-1}(z^* +v^*)\big)\big]
  	 		\end{eqnarray}
  	 		and
  	 		 		\begin{eqnarray}\label{BT1'}
  	 		\partial^\infty h(\bar w)= \bigcup_{(w^*, z^*)\in \partial^\infty f(\bar z, \bar w)}\;\bigcup_{v^*\in N(\bar z;
  	 			\Omega)}\big[ w^*+ T^*\big((M^*) ^{-1}(z^* +v^*)\big)\big],
  	 		\end{eqnarray}
  	 		where $\big(M^*) ^{-1}(z^* +v^*)=\{x^*\in X^*\mid M^*x^*=z^* +v^*\}$. 		
  	 	\end{theorem}
  	 	\begin{proof} (This proof is based on Theorem \ref{theoremAnYen}, Lemma \ref{bd1}, and Lemma \ref{bd2}.) We apply Theorem \ref{theoremAnYen} to the case where $w,z$, $f(z,w)$, $H(w)\cap \Omega$ and $h(w)$ play, respectively, the roles of $x,y$, $\varphi(x,y)$, $G(x)$ and $\mu(x)$. By the assumptions of the theorem, $f$ is continuous at $( \bar w, \bar z) \in  (W \times \Omega) \cap {\rm gph}\,H$. Hence, the regularity  condition (ii) of Theorem~\ref{theoremAnYen} is satisfied. Therefore,
  	 		\begin{align}\label{subdifferential1}
  	 		\partial h(\bar w) = \bigcup\limits_{(w^*,z^*) \in \partial f(\bar z, \bar w)}   \big\{w^* + D^*\widetilde{G}( \bar w, \bar z)(z^*) \big\}
  	 		\end{align}
  	 		and
  	 		\begin{align}\label{subdifferential2}
  	 		\partial^\infty h(\bar w) = \bigcup\limits_{(w^*,z^*) \in \partial^\infty f(\bar z, \bar w)}   \big\{w^* + D^*\widetilde{G}( \bar w, \bar z)(z^*) \big\},
  	 		\end{align}
  	 		where $\widetilde{G}(w):= H(w) \cap \Omega$ for all $w \in W$. Clearly, ${\rm gph}\,\widetilde{G}=(W \times \Omega) \cap {\rm gph}\,H.$  Let us show that
  	 		\begin{align}\label{coderivative}
  	 		D^*\widetilde{G}( \bar w, \bar z)(z^*) = \bigcup\limits_{v^*\in N(\bar z; \Omega)}   \big\{T^*[(M^*)^{-1}(z^*+v^*)] \big\}.
  	 		\end{align}
  	 		By the definition of coderivative,
  	 		\begin{align*}
  	 		D^*\widetilde{G}( \bar w, \bar z)(z^*) &= \big\{\widetilde{w}^* \in W^* \mid (\widetilde{w}^*, -z^*)\in  N ((\bar w, \bar z); {\rm gph}\, \widetilde{G})\}\\
  	 		&=\big\{\widetilde{w}^* \in W^* \mid (\widetilde{w}^*, -z^*)\in  N ((\bar w, \bar z); (W \times \Omega) \cap {\rm gph}\,H\big)\}.
  	 		\end{align*}
  	 		So, the assumptions made allow us to use formula \eqref{decomp} in Lemma \ref{bd2} to have
  	 		\begin{align*}
  	 		D^*\widetilde{G}( \bar w, \bar z)(z^*) &= \big\{\widetilde{w}^* \in W^* \mid (\widetilde{w}^*, -z^*)\in  \{0\} \times N(\bar z, \Omega) + N((\bar w, \bar z); {\rm gph}\, H)\big\}\\
  	 		&=\bigcup\limits_{v^*\in N(\bar z; \Omega)} \big\{\widetilde{w}^* \in W^* \mid (\widetilde{w}^*, -z^*)-(0,v^*)\in N((\bar w, \bar z); {\rm gph}\, H)\big\}\\
  	 		&=\bigcup\limits_{v^*\in N(\bar z; \Omega)} \big\{\widetilde{w}^* \in W^* \mid (\widetilde{w}^*, -z^*-v^* )\in N((\bar w, \bar z); {\rm gph}\, H)\big\}.
  	 		\end{align*}
  	 		Furthermore, as $\Phi$ has closed range, \eqref{normalcone2} is valid. Hence, $\widetilde{w}^* \in D^*\widetilde{G}( \bar w, \bar z)(z^*)$ if and only if there exist $v^*\in N(\bar z; \Omega)$ and $x^*\in X^*$ such that
  	 		$(\widetilde{w}^*, -z^*-v^* )=(-T^*x^*, M^*x^*)$. It follows that $x^*\in (M^*)^{-1}(-z^*-v^*)$ and $\widetilde{w}^*=-T^*x^*$. Therefore, $\widetilde{w}^* \in D^*\widetilde{G}( \bar w, \bar z)(z^*)$ if and only if
  	 		$$\widetilde{w}^*\in T^*[ (M^*)^{-1}(z^*+v^*)]$$
  	 		for some $z^* \in N(\bar z; \Omega)$. Thus, the equality \eqref{coderivative} has been proved.
  	 		
  	 		Combining \eqref{subdifferential1} with \eqref{coderivative}, we obtain \eqref{BT1}. Finally, we can easily get the equality \eqref{BT1'} from \eqref{subdifferential2} and \eqref{coderivative}.
  	 	  	 	$\hfill\Box$\end{proof}
  	 	
  	 \section{Differential Stability of the Control Problem}\markboth{\centerline{\it Main Results}} {\centerline{\it D. T. V.~An and N. T.~Toan
  		}} 

  	Based on Theorem \ref{asprogramingproblem}, we can obtain formulas for computing or estimating the subdifferential and singular subdifferential of the optimal value function $V(w)$ of the parametric control problem \eqref{problem1}--\eqref{problem4}.
  	
  	\medskip
  	In the notation of Subsections~2.1 and~2.2, put $Z=X\times U$ and $K=C\times
  	 \widetilde{X}\times\Omega$ and note that $	V(w)$ can be expressed as
  	 \begin{align}\label{optimalvaluefunction}
  	 V(w)=\inf_{z\in G(w)\cap K} f(z,w),
  	 \end{align}
  	  where
  	 \begin{align}\label{constraintmultifunction}
  	 G(w)=\big\{z=(x,u)\in Z \,| \, Mz=Tw\big\}
  	 \end{align}
  	 with $M:Z\to \widetilde{X}$ and $T: W\to
  	 \widetilde{X}$ are defined, respectively, by
  	 \begin{eqnarray}\nonumber
  	 \setcounter{MaxMatrixCols}{20} Mz=
  	 \left(
  	 \begin{array}{llllllllllll}
  	 -A_0 &I    &0 &0 &\ldots&0&0 &-B_0  &0  &0  &\ldots&0 \\
  	 0   &-A_1 &I &0 &\ldots&0 &0 &0    &-B_1 &0  &\ldots&0\\
  	 \vdots&\vdots&\vdots&\vdots&\vdots&\vdots&\vdots&\vdots&\vdots&\vdots&\vdots&\vdots\\
  	 0   &0    &0  &0 &\ldots&-A_{N-1}&I  &0  &0 &0
  	 &\ldots&-B_{N-1}
  	 \end{array}
  	 \right)
  	 \left(
  	 \begin{array}{ll}
  	 x_0\\
  	 x_1\\
  	 \vdots\\
  	 x_N\\
  	 u_0\\
  	 u_1\\
  	 \vdots\\
  	 u_{N-1}
  	 \end{array}
  	 \right),
  	 \end{eqnarray}
  	 
  	 \begin{eqnarray}\nonumber
  	 \setcounter{MaxMatrixCols}{20} Tw=
  	 \left(
  	 \begin{array}{ll}
  	 T_0w_0\\
  	 T_1w_1\\
  	 \vdots\\
  	 T_{N-1}w_{N-1}
  	 \end{array}
  	 \right).
  	 \end{eqnarray}
  	Then the problem \eqref{problem1}--\eqref{problem4} reduces to the mathematical programming problem \eqref{12b}. For every $\tilde x^*=(\tilde x_1^*, \tilde x_2^*,...,\tilde x_N^*) \in \widetilde{X}^*$, one has 
  	\begin{align}\label{Tconju}
  	T^*\tilde x^*=\big(T_0^*\tilde x_1^*, T_1^*\tilde x_2^*, \cdots,
  	 T_{N-1}^*\tilde x_N^*\big)\in W^*=W_0^*\times W_1^*\times\cdots\times W_{N-1}^*
  	 \end{align}
  	  and
  	 \begin{eqnarray}\label{Mconju}
  	 \setcounter{MaxMatrixCols}{20} M^*\tilde x^*=
  	 \left(
  	 \begin{array}{lllll}
  	 -A_0^*&0 &0 &\ldots&0\\
  	 I &-A_1^*&0 &\ldots&0\\
  	 0&I& &\ldots&0\\
  	 
  	 \vdots&\vdots&\vdots&\vdots&\vdots\\
  	 0 &0 &0 &\ldots&-A_{N-1}^*\\
  	 0 &0 &0 &\ldots &I\\
  	 -B_0^*&0 &0 &\ldots &0 \\
  	 0 &-B_1^*&0 &\ldots &0\\
  	 
  	 \vdots&\vdots&\vdots&\vdots&\vdots\\
  	 0 &0 &0 &\ldots &-B_{N-1}^*
  	 \end{array}
  	 \right)
  	 \left(
  	  \begin{array}{ll}
  	 \tilde x_1^*\\
  	 \tilde x_2^*\\
  	 \vdots\\
  	 \tilde x_N^*
  	 \end{array}
  	 \right),
  	 \end{eqnarray}
  	 where $T^*$, $M^*$, $A_i^*$, and $B_i^*$ are the adjoint operators of $T$, $M$, $A_i$, and $B_i$, respectively.
  	\begin{theorem}\label{Mainresult} Suppose that $h_k, \ k=0,1,\ldots, N$, are continuous and the interiors of $\Omega_k$, for $k=0,1,\ldots, N-1,$ are nonempty. Suppose in addition that the following conditions are satisfied:\par
  	  	{\rm (i)} ${\rm ker}\, T^* \subset {\rm ker}\,M^*;$ \par
  	{\rm (ii)} The operator $\Phi: W \times Z \rightarrow \widetilde{X}$ defined by $\Phi(w,z)=-Tw+Mz$ has closed range.\\
  	  	Then, if $\tilde w^*=(\tilde w_0^*,\tilde w_1^*,\ldots, \tilde w_{N-1}^*)\in \partial V(\bar w)$ then there exist $x_0^*\in N(\bar x_0; C)$, $\tilde x^*=(\tilde x_1^*, \tilde x_2^*,\ldots,\tilde x_N^*)\in {\widetilde{X}^*}$, and $u^*=(u_0^*, u_1^*,\ldots, u_{N-1}^*)\in N(\bar
  	  	   	 u;\Omega)$, such that
  	 \begin{align} \label{nondifferentiable}
  	 \begin{cases}
  	  \tilde x^*_N \in \partial h_N(\bar x_N),\\
  	   \tilde x^*_k \in  \partial _{x_k} h_k(\bar x_k, \bar u_k, \bar w_k)  + A^*_k \tilde x^*_{k+1}, \ k=1,2,...,N-1,\\
  	  x^*_0 \in -  \partial _{x_0} h_0(\bar x_0, \bar u_0, \bar w_0)  -A^*_0 \tilde x^*_1  ,\\
  	u^*_k \in - \partial _{u_k} h_k(\bar x_k, \bar u_k, \bar w_k) -B^*_k \tilde x^*_{k+1}  , \ k=0,1,...,N-1,\\
  	  \tilde w^*_k \in \partial _{w_k} h_k(\bar x_k, \bar u_k, \bar w_k) +T^*_k\tilde x^*_{k+1}, \ k=0,1,...,N-1.
  	   	   	 \end{cases}
  	  \end{align}
  	   
  	\end{theorem}
  	
  	\begin{proof}
  	Since the functions $h_k, \ k=0,1,\ldots, N,$ are continuous, the objective function $f$ of \eqref{12b} is continuous. Note that $V(w)$ coincides with the optimal value function $h(w)$ in \eqref{12} of \eqref{12b}. As $\Phi$ has closed range and ${\rm ker}\, T^* \subset {\rm ker}\,M^*$, applying Theorem \ref{asprogramingproblem} to \eqref{12b} yields
  		\begin{eqnarray}\label{formula_subdifferential}
  		\partial V(\bar w)= \bigcup_{(z^*, w^*)\in \partial f(\bar z, \bar w)}\;\bigcup_{v^*\in N(\bar z;K)}\big[ w^*+ T^*\big((M^*) ^{-1}(z^* +v^*)\big)\big].
  		\end{eqnarray}
  Consider the function $\tilde{h}_k: X \times U \times W \rightarrow \Bbb{R},\  k=0,1,...,N-1$, given by $\tilde{h}_k(x,u,w)={h}_k(x_k,u_k,w_k)$. Let $\tilde{h}_N: X \times U \times W \to \Bbb{R}$ be defined by $\tilde{h}_N(x,u,w)={h}_N(x_N).$ 
 From \eqref{formula_subdifferential} one has $ \tilde w^* \in \partial V(\bar w)$ if and only if there exist $(z_1^*,w_1^*) \in \partial f(\bar z, \bar w)$ and $v_1^*=(x^*,u^*) \in N(\bar z; K)$ such that
 $$\tilde{w}^* \in w_1^* + T^* ((M^*)^{-1} )(z_1^*+v_1^*).$$
 The last inclusion means that there exists $\tilde x^*=(\tilde x^*_1, \tilde x^*_2,...,\tilde x^*_N) \in \widetilde{X}^*$ such that 
 \begin{align}\label{32new}
 \begin{cases}
 M^* \tilde{x}^* = z_1^* +v_1^*,\\
 \tilde{w}^* \in w_1^* +T^*\tilde{x}^*. 
 \end{cases}
 \end{align}
 Denote by $\partial _z f(\bar z, \bar w), \partial _w f(\bar z, \bar w)$ the subdifferentials of $ f(., \bar w)$ at $\bar z$ and $ f(\bar z,.)$ at $\bar w$, respectively. We have
     	$\partial f(\bar z,\bar w) \subset \partial_{z} f (\bar z, \bar w)\times \partial_{w} f(\bar z, \bar w).$ Indeed, for any $(z_1^*,w_1^*)\in \partial f(\bar z, \bar w),$ 
     	$$\langle (z_1^*,w_1^*), (z,w)-(\bar z, \bar w) \rangle \le f(z,w)-f(\bar z, \bar w),\, \forall (z,w)\in Z \times W.$$
     	Hence, taking $w=\bar w$ yields
     	$\langle z_1^*, z-\bar z \rangle \le f(z,\bar w)-f(\bar z, \bar w)$ for every $z\in Z.$ So $z_1^* \in \partial_z f(\bar z, \bar w).$ Similarly, $w_1^* \in \partial_w f(\bar z, \bar w).$ As $(z_1^*,w_1^*) \in \partial f(\bar z, \bar w)$, in combination with \eqref{32new}, the preceding observation gives
  	\begin{align}\label{formulas_subdifferential2}
  	\begin{cases}
  	M^*\tilde{x}^* \in  \partial _z f(\bar z, \bar w)+v_1^*,\\
  	\tilde{w}^* \in  \partial_w f(\bar z, \bar w) +T^* \tilde{x}^*.
  	\end{cases}
  	\end{align}
   By the continuity of $\tilde{h}_k(.), \ k=0,1,\ldots, N$, applying the Moreau--Rockafellar Theorem \cite[p.~48]{Ioffe_Tihomirov_1979}, we have 
  	  \begin{equation}
  	    	\begin{split} \label{newformula}
  	  	\partial_z f(\bar z, \bar w)& = \partial_z \left( \sum\limits_{k=0}^{N} \tilde{h}_k\right)(\bar z, \bar w) =\sum\limits_{k=0}^{N}  \partial_z \tilde{h}_k(\bar z, \bar w)  	  	  \\
  	  	  & \subset \sum\limits_{k=0}^{N} \partial_{x_k} \tilde{h}_k (\bar x_k, \bar u_k, \bar w_k)\times \partial_{u_k} \tilde{h}_k (\bar x_k,\bar u_k, \bar w_k) .
  \end{split}
    	  	  	  \end{equation}
It is easy to see that
  \begin{align*}
  & \partial_{x_0} \tilde{h}_0(\bar x_0, \bar u_0, \bar w_0)=\partial _{x_0}h_0(\bar x_0, \bar u_0, \bar w_0) \times \{0\}\times ... \times \{0\},\\
  &\partial_{x_1} \tilde{h}_1(\bar x_1, \bar u_1, \bar w_1)=\{0\}\times \partial _{x_1}h_1(\bar x_1, \bar u_1, \bar w_1)\times ... \times \{0\},\\
  &...\\
  & \partial_{x_{N\!-\!1}} \tilde{h}_{N\!-\!1}\!(\bar x_{N\!-\!1}, \bar u_{N\!-\!1}, \bar w_{N\!-\!1}\!)\!=\!\{0\}\!\times \!...\!\times\!\partial _{x_{N\!-\!1}}h_{N\!-\!1}(\bar x_{N\!-\!1}, \bar u_{N\!-\!1}, \bar w_{N\!-\!1}) \!\times\! \{0\},\\
  &\partial _{x_N} \tilde{h}_N( \bar x_N)=\partial h_N(\bar x_N).
      	  	\end{align*}
 Similarly,
  \begin{align*}
    & \partial_{u_0} \tilde{h}_0(\bar x_0, \bar u_0, \bar w_0)=\partial _{u_0}h_0(\bar x_0, \bar u_0, \bar w_0) \times \{0\}\times ... \times \{0\},\\
    &\partial_{u_1} \tilde{h}_1(\bar x_1, \bar u_1, \bar w_1)=\{0\}\times \partial _{u_1}h_1(\bar x_1, \bar u_1, \bar w_1)\times ... \times \{0\},\\
    &...\\
    & \partial_{u_{N\!-\!1}} \tilde{h}_{N\!-\!1}\!(\bar x_{N\!-\!1}, \bar u_{N\!-\!1}, \bar w_{N\!-\!1}\!)\!=\!\{0\}\!\times \!...\!\times\!\partial _{u_{N\!-\!1}}h_{N\!-\!1}(\bar x_{N\!-\!1}, \bar u_{N\!-\!1}, \bar w_{N\!-\!1}) \!\times\! \{0\}.
        	  	\end{align*}    	  		  	   
   Hence
   \begin{equation}\label{33a}
   \begin{split}
    &\partial_z f(\bar z, \bar w)\\ & \subset  \partial_{x_0}h_0 (\bar x_0, \bar u_0, \bar w_0)\! \times\!...\!\times \! \partial_{x_{N\!-\!1}}h_{N\!-\!1} (\bar x_{N\!-\!1}, \bar u_{N\!-\!1}, \bar w_{N\!-\!1})\!\times\! \partial h_N (\bar x_{N}) \\
    & \quad \times  \partial_{u_0}h_0 (\bar x_0, \bar u_0, \bar w_0)\times...\times \partial_{u_{N-1}}h_{N-1} (\bar x_{N-1}, \bar u_{N-1}, \bar w_{N-1}).
    \end{split}
    \end{equation} 
   In the same manner, we obtain
  	\begin{equation}\label{33b}
  	\begin{split}
    	  	\partial_w f(\bar z, \bar w) &\subset \sum\limits_{k=0}^{N-1} \partial_{w_k} \tilde{h}_k (\bar x_k, \bar u_k, \bar w_k)
    	  \\
    	  &\subset {\partial}_{w_0} h_0 (\bar x_0, \bar u_0, \bar w_0) \times...\times{\partial}_{w_{N-1}} h_{N-1} (\bar x_{N-1}, \bar u_{N-1}, \bar w_{N-1}) .
   	  	  	\end{split}
   	  	  	\end{equation}
   	Since 
   	\begin{align*}
   	v_1^*\in N(\bar z; K)&=N(\bar x_0; C)\times\{0_{\tilde{X}^*}\}\times N(\bar
     	  	u;\Omega)\\
   &=N(\bar x_0; C)\times\{0_{\tilde{X}^*}\}\times N(\bar
     	  	     	  	u_0;\Omega_0)\times ...\times N(\bar u_{N-1}; \Omega_{N-1}),
    \end{align*} 	  	
  there exist $x_0^*\in N(\bar x_0; C)$ and $ u^*=(u_0^*,
     	  	u_1^*,\ldots,u_{N-1}^*)$ with $u_k^*\in N(\bar u_k; \Omega_k)$
     	  	$(k=0,1,\ldots,N-1)$ such that $v_1^*=(x_0^*,0,u^*)$.
  Therefore, from the first inclusion in \eqref{formulas_subdifferential2} and from \eqref{Mconju}, \eqref{33a}, we get
  	\begin{align*}
    &\left(
    	    \begin{array}{lllll}
    	-A_0^* &0 &0 &\ldots &0\\
    	I &-A_1^* &0 &\ldots &0\\
    	0 &I &0 &\ldots &0\\
    	\vdots &\vdots&\vdots&\vdots&\vdots\\
    	0 &0 &0 &\ldots &-A_{N-1}^*\\
    	0 &0 &0 &\ldots &I
    	\end{array}
    	  \right)
    	\left(
    	    \begin{array}{ll}
    	\tilde x_1^*\\
    	\tilde x_2^*\\
    	\vdots\\
    	\tilde x_N^*
    	\end{array}
    	 \right)
    	 \\
    	 & \in  \big(\partial_{x_0} h_0(\bar x_0,\bar u_0,\bar w_0)+x_0^*\big)\times \partial_{x_1} h_1(\bar x_1,\bar u_1,\bar w_1)\times...\times \\
    	 &\quad \quad \quad \partial_{x_{N-1}} h_{N-1} (\bar x_{N-1}, \bar u_{N-1},\bar w_{N-1}) \times \partial h_N(\bar x_N)
    	\end{align*}
    	and
    \begin{align*}
       	   &	\left(
    	    	    \begin{array}{lllll}
    	    	-B_0^* &0 &0 &\ldots &0\\
    	    	0 &-B_1^* &0 &\ldots &0\\
    	       	    	\vdots &\vdots&\vdots&\vdots&\vdots\\
    	    	    	0 &0 &0 &\ldots  &-B^*_{N-1}
    	    	\end{array}
    	    	  \right)
    	    	   	    	\left(
    	    	    \begin{array}{ll}
    	    	\tilde x_1^*\\
    	    	\tilde x_2^*\\
    	    	\vdots\\
    	    	\tilde x_N^*
    	    	\end{array}
    	    	  \right) \\
    	    	 & \in  \big(\partial_{u_0} h_0(\bar x_0,\bar u_0,\bar w_0)+u_0^*\big )\times...\times
    	    	  \big (\partial_{u_{N-1}} h_{N-1}(\bar x_{N-1}, \bar u_{N-1},
    	    	  \bar w_{N-1}) +u^*_{N-1} \big ).
    	    \end{align*}
   This implies that
 \begin{equation}\label{33c}
 \begin{split}
  \begin{cases}
 -x_0^* \in A_0^* \tilde x_1^* +  \partial _{x_0} h_0 (\bar x_0,\bar u_0, \bar w_0),\\
 \tilde{x}^*_{1} \in A_1^* \tilde{x}^*_2 +   \partial _{x_1} h_1 (\bar x_1,\bar u_1, \bar w_1),\\
 ...\\
   \tilde{x}^*_{N-1} \in A_{N-1}^* \tilde{x}^*_N +   \partial _{x_{N-1}} h_{N-1} (\bar x_{N-1},\bar u_{N-1}, \bar w_{N-1}),\\
   \tilde{x}^*_N \in \partial h_N(\bar x_N)
 \end{cases}
 \end{split}
  \end{equation}
 and 
 \begin{align}\label{33d}
  -B^*_k \tilde{x}^*_{k+1} \in  \partial _{u_k} h_k (\bar x_k, \bar u_k, \bar w_k)+u^*_k, \ k=0,1,...,N-1.
 \end{align} 
 Now we can derive from the second inclusion in \eqref{formulas_subdifferential2} and from \eqref{Tconju}, \eqref{33b}, the following
  \begin{align}
  \label{33e}
   \tilde{w}^*_k \in  \partial _{w_k} h_k (\bar x_k, \bar u_k, \bar w_k)+T^*_k \tilde{x}^*_{k+1}, \ k=0,1,...,N-1.
    \end{align}
  Combining \eqref{33c}--\eqref{33e}, we obtain \eqref{nondifferentiable}. The proof is complete.
  	$\hfill\Box$
   	\end{proof}
  
  	\begin{theorem}\label{Frechetcase}
  	Under the assumptions of Theorem \ref{Mainresult}, suppose additionally that the functions $h_k,$ for $ k=0,1,\ldots,N$, are Fr\'echet differentiable. Then, a vector  $\tilde w^*=(\tilde w_0^*,\tilde w_1^*,\ldots, \tilde w_{N-1}^*)\in W^*$ belongs to  $\partial V(\bar w)$ if and only if there exist $x_0^*\in N(\bar x_0; C)$, $\tilde x^*=(\tilde x_1^*, \tilde x_2^*,\ldots,\tilde x_N^*)\in {\widetilde{X}^*}$ and $u^*=(u_0^*, u_1^*,\ldots, u_{N-1}^*)\in N(\bar
    	   	  	   	  	   	 u;\Omega)$ such that
    	   	  	 \begin{eqnarray}\label{Frechet_case}
    	\left\{
    	  \begin{array}{ll} \tilde x_N^*=\nabla h_N(\bar
    	  x_N),\\ \tilde x_k^*=\nabla_{x_k}h_k(\bar x_k, \bar u_k, \bar
    	  w_k)+A_k^*\tilde x_{k+1}^*,\ k=1,2,\ldots, N-1,\\
    	  	x_0^*=-\nabla_{x_0}h_0(\bar x_0, \bar u_0, \bar
    	  	w_0)-A_0^*\tilde x_1^*,\\
    	  	u_k^*=-\nabla_{u_k}h_k(\bar x_k, \bar u_k, \bar
    	  	w_k)-B_k^*\tilde x_{k+1}^*,\ k=0,1,\ldots, N-1,\\
    	\tilde w_k^*=\nabla_{w_k}h_k(\bar x_k, \bar u_k, \bar
    	w_k)+T_k^* \tilde x_{k+1}^*, \ k=0,1,\ldots, N-1,\\
    	 \end{array}
    	\right.
    	\end{eqnarray}
    	where $\nabla h_N(\bar x_N),\nabla_{x_k}h_k(\bar x_k, \bar u_k, \bar
    	w_k)$, $\nabla_{u_k}h_k(\bar x_k, \bar u_k, \bar
    	w_k)$ and $\nabla_{w_k}h_k(\bar x_k, \bar u_k, \bar
    	w_k)$, respectively, stand for the Fr\'echet derivatives of $h_N(.),$ $ h_k(., \bar u_k,\bar w_k)$, $h_k(\bar x_k,., \bar
    	w_k)$ and $h_k(\bar x_k, \bar
    	u_k,.)$ at $\bar x_N, \bar
    	x_k$, $\bar u_k$, and $\bar w_k$. 
    	\end{theorem}
\begin{proof}
It is well-known that if $\varphi : Y \to \overline{\Bbb{R}}$ is a convex function defined on a normed space $Y$ and $\varphi$ is Fr\'echet differentiable at $\bar y \in Y$, then $\partial \varphi (\bar y)=\{\nabla \varphi(\bar y)\}$ (see e.g. \cite [p.~197--198]{Ioffe_Tihomirov_1979}). Hence, since $h_k$, $k=0,1,...,N,$ are Fr\'echet differentiable by our assumptions, the inclusions in \eqref{formulas_subdifferential2}--\eqref{33b} become equalities. Namely, we have
	\begin{align*}
  	\begin{cases}
  	M^*\tilde{x}^* = \nabla _z f(\bar z, \bar w)+v_1^*\\
  	\tilde{w}^* =  \nabla_w f(\bar z, \bar w) +T^* \tilde{x}^*,
  	\end{cases}
  	\end{align*}
  	
  	 \begin{equation*}
  	\nabla_z f(\bar z, \bar w) =
   \sum\limits_{k=0}^{N} \big (\nabla_{x_k} \tilde{h}_k (\bar x_k, \bar u_k, \bar w_k), \nabla_{u_k} \tilde{h}_k (\bar x_k,\bar u_k, \bar w_k) \big ),
     \end{equation*}
  \begin{equation*}
    \begin{split}
     \nabla_z f(\bar z, \bar w)&= \big (\nabla_{x_0}h_0 (\bar x_0, \bar u_0, \bar w_0)\!,\!...\!, \! \nabla_{x_{N\!-\!1}}h_{N\!-\!1} (\bar x_{N\!-\!1}, \bar u_{N\!-\!1}, \bar w_{N\!-\!1}), \nabla h_N (\bar x_{N}), \\
     & \quad \quad  \nabla_{u_0}h_0 (\bar x_0, \bar u_0, \bar w_0),..., \nabla_{u_{N-1}}h_{N-1} (\bar x_{N-1}, \bar u_{N-1}, \bar w_{N-1})\big ),
     \end{split}
     \end{equation*}    
     and 
   \begin{equation*}
  	\begin{split}
  	\nabla_w f(\bar z, \bar w) = \big ({\nabla}_{w_0} h_0 (\bar x_0, \bar u_0, \bar w_0) ,...,{\nabla}_{w_{N-1}} h_{N-1} (\bar x_{N-1}, \bar u_{N-1}, \bar w_{N-1})\big ) .
  	\end{split}
  	\end{equation*}    
 Consequently, by the proof of Theorem \ref{Mainresult} we can conclude that  a vector  $\tilde w^*=(\tilde w_0^*,\tilde w_1^*,\ldots, \tilde w_{N-1}^*)\in W^*$ belongs to  $\partial V(\bar w)$ if and only if there exist $x_0^*\in N(\bar x_0; C)$, $\tilde x^*=(\tilde x_1^*, \tilde x_2^*,\ldots,\tilde x_N^*)\in {\widetilde{X}^*}$ and $u^*=(u_0^*, u_1^*,\ldots, u_{N-1}^*)\in N(\bar
     	   	  	   	  	   	 u;\Omega)$ such that \eqref {Frechet_case} is satisfied.     	   	  	   	  	   	 
$\hfill\Box$  	
\end{proof}
	
  	\begin{theorem}\label{Mainresult2}
  	Under the assumptions of Theorem \ref{Mainresult}, we have 
  	$$\partial^\infty V(\bar w)=\{0_{W^*}\}.$$
  	\end{theorem}
  	\begin{proof}
  Similarly as in the proof of Theorem \ref{Mainresult}, applying Theorem \ref{asprogramingproblem} to \eqref{BT1'}, we get
  			\begin{eqnarray}\label{BT1b}
  		  		\partial^\infty V(\bar w)= \bigcup_{(z^*, w^*)\in \partial^\infty f(\bar z, \bar w)}\;\bigcup_{v^*\in N(\bar z;K)}\big[ w^*+ T^*\big((M^*) ^{-1}(z^* +v^*)\big)\big].
  		  		\end{eqnarray}
  		Since ${\rm dom}\,f=Z\times W$ and $\partial^\infty f(\bar z, \bar w)= N((\bar z, \bar w); {\rm dom }\, f)$ by \cite[Proposition~4.2]{AnYen}, we have $\partial^\infty f(\bar z, \bar w)=\{(0_{Z^*},0_{W^*})\}$. Therefore, from \eqref{BT1b} it follows that
  	  	  \begin{eqnarray}\label{BT1b1}
  	  	  \partial^\infty V(\bar w)= \bigcup_{v^*\in N(\bar z;K)}\big[T^*\big((M^*) ^{-1}(v^*)\big)\big].
  	  	  \end{eqnarray}
  	  	    	  Thus, $\tilde{w}^* $ belongs to $ \partial^\infty V( \bar w)$ if and only if there exist $v^*\in N(\bar{z};K)$ and $\tilde x^*=(\tilde x_1^*, \tilde x_2^*,\ldots,\tilde x_N^*)\in {\widetilde{X}^*}$ such that
  	  	    	  \begin{align*}
   	  M^* \tilde x^*=v^*,	  	 \   	  \tilde w^*=T^* \tilde x^*.  	  	    	  \end{align*}
 As $v^*\in N(\bar z; K)=N(\bar x_0; C)\times\{0_{\tilde{X}^*}\}\times N(\bar u_0;\Omega_0)\times...\times N(\bar u_{N-1}; \Omega_{N-1}),$ we can find $x_0^*\in N(\bar x_0; C)$ and $ u^*=(u_0^*,
  	  	      	u_1^*,\ldots,u_{N-1}^*)$ with $u_k^*\in N(\bar u_k; \Omega_k)$ for
  	  	      	$k=0,1,\ldots,N-1$. So $v^*=(x_0^*,0_{\tilde{X}^*}, u_0^*,...,u^*_{N-1}).$  By \eqref{Tconju} and \eqref{Mconju}, we see that $\tilde w^*\in \partial^\infty V(\bar w)$ if and only if there exist $x_0^*\in N(\bar{x}_0;C)$  and $\tilde x^*=(\tilde x_1^*, \tilde x_2^*,\ldots,\tilde x_N^*)\in {\widetilde{X}^*}$, and $u^*=(u_0^*,u_1^*,...,u^*_{N-1})$ with $u_k^*\in N(\bar{u}_k; \Omega_k)$, $k=0,1,...,N-1,$ such that
  	  	      	\begin{align}\label{singular_subdifferential}
  	  	      	\begin{cases}  \tilde x^*_N=0,\\
  	  	      	\tilde x^*_k=A_k^* \tilde x^*_{k+1},\ k=1,\dots,N-1,\\
  	  	      	x_0^*=-A_0^* \tilde x^*_1,\\ 
  	  	      	u^*_k=-B^*_k \tilde x^*_{k+1},\ k=0,1,\dots,N-1,\\
  	  	      	w^*_k=T^*_k \tilde x^*_{k+1},\ k=0,1,\dots,N-1.
  	  	      	\end{cases}
  	  	      	\end{align}  	
  	  	      From \eqref{singular_subdifferential} we can easily deduce that $\partial ^\infty V(\bar w)=\{0_{W^*}\}.$
  	  	      	$\hfill\Box$
  	\end{proof}
    	
   	Let us describe a typical situation where the assumptions of Theorem \ref{Mainresult} are automatically satisfied.
 \begin{remark} If $T_0, T_1,...,T_{N-1}$ are surjective, then the operator  $T: W \rightarrow \widetilde{X}$ is surjective too. Hence ${\rm{ker}}\, T^*=\{0\},$ and, therefore, condition \rm{(i)} in Theorem \ref{Mainresult} is satisfied. Moreover, condition \rm{(ii)} of that theorem is also fulfilled, because $\rm{rge}\, \Phi=\widetilde{X}.$
 \end{remark} 	
  	
  \section{Applications}\markboth{\centerline{\it Main Results}} {\centerline{\it D. T. V.~An and N. T.~Toan
   	   		}} 	
  In this section we apply the obtained results to some examples. First, we give an auxiliary result related to a convex optimization problem under linear constraints.
  
  \medskip
  Let $X$ be a Banach space with the dual denoted by $X^*$. Consider the problem
  $$\min \big\{\varphi(x) \mid \langle a_i,x \rangle \le \alpha_i,\, \langle b_j,x \rangle = \beta_j, \, i=1,\dots,m, \, j=1,\dots,k \big\}, \ \ \ \ \ \quad  (P)$$
where $\varphi: X \to \Bbb{R}$ is a continuous convex function, $a_i, b_j \in X^*$, $i=1,\dots,m,\, j=1,\dots,k.$  

\medskip
Denote by $\Omega$ and Sol$\,(P)$, respectively, the constraint set and the solution set of $(P)$.

\medskip  
The following statement is a Farkas lemma for infinite dimensional vector spaces.
\begin{lemma}\label{Farkas_lemma}{\rm{(See \cite[Lemma~1]{Bartl_2008})}}
Let $W$ be a vector space over $\Bbb{R}.$ Let $A: W \rightarrow \Bbb{R}^m$ be a linear mapping and $\gamma : W \rightarrow \Bbb{R}$ be a linear functional. Suppose that $A$ is represented in the form $A=(\alpha_i)_i^m$, where each  $\alpha_i:W\to \Bbb{R}$ is a linear functional (i.e., 
for each $x\in W$, $A(x)$ is a column vector whose $i-th$ component is $\alpha_i(x)$, for $i=1,\dots,m$). Then, the inequality $\gamma(x) \le 0$ is a consequence of the inequalities system
$$\alpha_1(x) \le 0, \ \alpha_2(x) \le 0,\dots, \ \alpha_m(x) \le 0$$
if and only if there exist nonnegative real numbers $\lambda_1, \lambda_2,\dots,\lambda_m \ge 0$ such that
$$\gamma=\lambda_1\alpha_1+\dots+\lambda_m \alpha_m.$$
\end{lemma}

Based on Lemma \ref{Farkas_lemma} and a standard Fermat rule for convex programs, one can obtain the next proposition on necessary and sufficient optimality conditions for $(P)$, which is very useful for dealing with \eqref{problem1}--\eqref{problem4} when $C$ and $\Omega_i, \, i=0,1,...,N-1$, are polyhedral convex sets. For clarity of our presentation, we provide here a detailed proof of this result.
\begin{proposition} \label{Pro_Auxilary} 
For a point $\bar x \in \Omega$ to be a solution of $(P)$, it is necessary and sufficient, that there exist $\lambda_i \ge 0,$ $i=1,\dots,m$ and $\mu_j \in \Bbb{R},$ $j=1,\dots,k,$ such that
\\ {\rm {(a)}} $0 \in \partial \varphi (\bar x) + \sum\limits_{i=1}^m \lambda_i a_i+\sum\limits_{j=1}^k \mu_j b_j ;$
\\ {\rm {(b)}} $\lambda_i (\langle a_i,\bar x \rangle -\alpha_i)=0,$ $i=1,\dots,m.$
\end{proposition} 

\begin{proof}
Let $\bar x \in \Omega$ be given arbitrary. Note that $(P)$ can be written in the form
$$\min \big\{\varphi(x)\mid x \in \Omega\big\}.$$
By  \cite[Proposition 2, p.~81]{Ioffe_Tihomirov_1979}, $\bar x\in $ Sol$\,(P)$ if and only if 
\begin{align}\label{Fermat}
0\in \partial \varphi (\bar x) +N(\bar x; \Omega).
\end{align}
We now show that 
\begin{equation}\label{cone_normal}
\begin{split}
&N(\bar x; \Omega)\\
&=\left \{\sum\limits_{i\in I(\bar x) } \lambda_i a_i+\sum\limits_{j=1}^k \mu_j b_j\!\mid \!\, \lambda_i \ge 0, \, i=1,...,m, \ \mu_j \in \Bbb{R}, \, j=1,...,k\right \},
\end{split}
\end{equation}
where $I(\bar x)=\{i \mid \langle a_i, \bar x \rangle=\alpha_i, \, i=1,\dots,m\}.$

Take any $x^* \in N(\bar x; \Omega).$ Let
$v \in X$ be such that $\langle a_i, v \rangle \le 0$ for $i \in I(\bar x)$, $\langle b_j, v \rangle \le 0$ and $\langle -b_j, v \rangle \le 0$  for $j=1,\dots,k$. If $t>0$ is chosen small enough, then 
$$\langle a_i, \bar x +tv \rangle =\langle a_i, \bar x \rangle + t \langle a_i, v \rangle \le \alpha_i,\ \;  i=1,\dots,m,$$
and 
$$\langle b_j, \bar x +tv \rangle =\beta_j,\ \; j=1,\dots,k.$$
Thus $\bar x +tv \in \Omega$; so we have
$$0 \ge \langle x^*, (\bar x +tv)-\bar x \rangle=t \langle x^*, v \rangle.$$
It follows that, $\langle x^*, v \rangle \le 0.$ Hence, the inequality $\langle x^*, v \rangle \le 0$ is a consequence of the inequalities system 
$$
\begin{cases}
\langle a_i, v \rangle \le 0, \ i \in I(\bar x),\\
\langle b_j, v \rangle \le 0, \ j=1,\dots,k,\\
\langle -b_j, v \rangle \le 0,\ j=1,\dots,k.
\end{cases}
$$
By Lemma \ref{Farkas_lemma}, there exist $\lambda_i \ge 0, \, i \in I(\bar x)$, $\mu_j^1 \ge 0$ and $\mu_j^2 \ge 0$, $j=1,\dots,k$, such that
$$ x^*=\sum\limits_{i \in I(\bar x)} \lambda_i a_i +\sum\limits_{j=1}^k \mu_j^1 b_j +\sum\limits_{j=1}^k \mu_j^2(-b_j).$$
Setting $\mu_j:=\mu_j^1-\mu_j^2,$ from the last equality, we can deduce that $x^*$ belongs to the right-hand-side of \eqref{cone_normal}.

Now, suppose $x^*= \sum\limits_{i\in I(\bar x)} \lambda_i a_i +\sum\limits_{j=1}^k \mu_j b_j,$ where $\lambda_i \ge 0$, for every $i \in I(\bar x)$, $\mu_j \in \Bbb{R}$, for $j=1,\dots,k$.
Given any $x \in \Omega$, we have 
$$
\langle a_i, x -\bar x \rangle = \langle a_i, x \rangle -\langle a_i, \bar x \rangle \le \alpha_i-\alpha_i=0, \ \; i \in I(\bar x).
$$
Therefore 
\begin{align*}
\langle x^*, x - \bar x \rangle &= \left\langle 
\sum\limits_{i\in I(\bar x) } \lambda_i a_i +\sum\limits_{j=1}^k \mu_j b_j, x- \bar x  \right\rangle \\
&= \sum\limits_{i \in I(\bar x)} \lambda_i \langle a_i, x- \bar x \rangle\\
& \le 0.
\end{align*}
It follows that $x^* \in N(\bar x; \Omega)$.
Combining \eqref{Fermat} and \eqref{cone_normal}, we obtain the assertion of the proposition. 	$\hfill\Box$
\end{proof}

  The next example is designed to show how Theorem \ref{Frechetcase} can work for parametric optimal control problems with differentiable objective functions.

  \begin{example} Let $N=1$, $X_0=\Bbb{R}$, $X_1=\Bbb{R}$, $U_0=\Bbb{R},$ $W_0=\Bbb{R},$ $\Omega_0=[-1,+\infty)$ and $C=(-\infty, 2]$. Let $A_0: X_0\to X_1$, $B_0: U_0\to X_1$, $T_0: W_0\to X_1$ be defined by $A_0x_0=x_0$, $B_0u_0=-u_0$, and $T_0w_0=2w_0$. Furthermore, let 
  $h_0: X_0\times U_0 \times W_0 \to \Bbb{R}$ and $h_1: X_1 \to \Bbb{R}$ be given, respectively, by
 \begin{align*}
 &h_0(x_0,u_0,w_0)=x_0^2+x_0u_0+u_0^2 +\dfrac{1}{2}w_0,\\
 &h_1(x_1)=(x_1+1)^2.
 \end{align*}
 Consider the control problem \eqref{problem1}--\eqref{problem4} and choose $\bar w=0$ belonging to $W=W_0=\Bbb{R}.$ Then, the problem \eqref{problem1}--\eqref{problem4} becomes
 $$
  \begin{cases}
 f(x,u,\bar w)=x_0^2+x_0u_0+u_0^2+(x_1+1)^2 \to \rm{inf},\\
 x_1=x_0-u_0,\\
 x_0\le 2,\\
 u_0\ge -1.
 \end{cases}
 \quad \quad \quad \quad \quad \quad  \quad \quad (P_1)
 $$
 Using Proposition \ref{Pro_Auxilary}, it is easy to show that $(\bar x_0,\bar x_1,\bar u_0)=
 \left(-\dfrac{2}{5},-\dfrac{4}{5},\dfrac{2}{5}\right)$ is the unique solution of $(P_1)$. Hence $\bar x=(\bar x_0, \bar x_1)=\left(-\dfrac{2}{5},-\dfrac{4}{5}\right)$ and $\bar u=\bar u_0=\dfrac{2}{5},$ we have
 $(\bar x, \bar u) \in S(\bar w)$. 
 
  Clearly, the mapping $\Phi: \Bbb{R}\times \Bbb{R}^3 \to \Bbb{R}$ given by 
  $$\Phi(w,z)=Mz-Tw=-x_0+x_1+u_0-2w_0$$
   has closed range and ${\rm{ker}}\, T^*= {\rm{ker}}\, M^*=\{0\}$. Hence, by Theorem \ref{Frechetcase}, $w^*_0\in \partial V(\bar w)$ if and only if there exist $x_0^*\in N(\bar x_0; C)$, $\tilde x_1^*\in \Bbb{R}$, and $u_0^*\in N(\bar u_0;\Omega_0)$ such that
  \begin{eqnarray}\label{exampleFrechet}
     	\left\{
     	  \begin{array}{ll} \tilde x_1^*=\nabla h_1(\bar
     	  x_1),\\ 
     	   x_0^*=-\nabla_{x_0}h_0(\bar x_0, \bar u_0, \bar
     	  	w_0)-A_0^*\tilde x_1^*,\\
     	  	u_0^*=-\nabla_{u_0}h_0(\bar x_0, \bar u_0, \bar
     	  	w_0)-B_0^*\tilde x_{1}^*,\\
     	 w_0^*=\nabla_{w_0}h_0(\bar x_0, \bar u_0, \bar
     	w_0)+T_0^* \tilde x_{1}^*.\\
     	 \end{array}
     	\right.
     	\end{eqnarray}
   We have
    \begin{align*}
     & \nabla h_1(\bar x_1)=\nabla h_1\left( -\dfrac{4}{5}\right)=\dfrac{2}{5} ,\\
    & \nabla_{x_0}h_0(\bar x_0, \bar u_0, \bar
   	w_0)=\nabla_{x_0}h_0\left(-\dfrac{2}{5}, \dfrac{2}{5}, 0\right) =-\dfrac{2}{5} ,\\
      	& \nabla_{u_0}h_0(\bar x_0, \bar u_0, \bar 	w_0)=\nabla_{u_0}h_0\left(-\dfrac{2}{5}, \dfrac{2}{5}, 0\right) = \dfrac{2}{5},\\
      		& \nabla_{w_0}h_0(\bar x_0, \bar u_0, \bar
      	   	w_0)=\nabla_{w_0}h_0 \left(-\dfrac{2}{5}, \dfrac{2}{5}, 0\right)=\dfrac{1}{2},
   	\end{align*}
   $N(\bar x_0; C)=N\left(-\dfrac{2}{5};(-\infty,2] \right)=\{0\}$, $N(\bar u_0; \Omega_0)=N\left(\dfrac{2}{5};[-1,+\infty)\right)=\{0\}$, $A^*_0=1,$ $ B^*_0=-1,$ and $ T^*_0=2.$
   Thus, from \eqref{exampleFrechet} we have $ \tilde x_1^*=\dfrac{2}{5}, \, x_0^*=0, \, u_0^*=0, \, w^*_0=\dfrac{13}{10}. $ Hence $\partial V(\bar w)=\left\{\dfrac{13}{10} \right\}.$
   
  \end{example}

  Next, we give an example to illustrate the result of Theorem \ref{Mainresult}, where $h_0,...,h_N$ are not required to be differentiable.
  	 \begin{example}\rm Choose $N=2$, $X_0=X_1=X_2=\Bbb{R},$ 
  	 	$ U_0=U_1=\Bbb{R},$  $W_0=W_1=\Bbb{R}$, $C=(-\infty, 1]$, and $\Omega_0=\Omega_1=\Bbb{R}$. 
 Let $A_0: X_0\to X_1,$ $B_0:U_0\to X_1$, $T_0: W_0\to X_1$, $A_1: X_1\to X_2$, $B_1: U_1\to X_2$ and $T_1: W_1 \to X_2$ be given by $A_0x_0=-x_0$, $B_0u_0=0$, $T_0w_0=-w_0$, $A_1x_1=x_1$, $B_1u_1=-u_1$, and $T_1w_1=w_1$. Furthermore, define $h_0:X_0\times U_0\times W_0\to \Bbb{R}$, $h_1: X_1\times U_1 \times W_1 \to \Bbb{R}$, and $h_2: X_2 \to \Bbb{R}$ by
  	 	 	\begin{align*}
  	 &	h_0(x_0,u_0,w_0)=(x_0+u_0)^2 +\dfrac{1}{2}w_0^2,\\
  	 &	h_1(x_1,u_1,w_1)=|x_1-1|+|w_1|,\\
  	 &	h_2(x_2)=|x_2|.
  	 	\end{align*}
Then, at the parameter $\bar w=(\bar w_0, \bar w_1)=(0,0)$, the problem \eqref{problem1}--\eqref{problem4} in the following
 $$
  \begin{cases}
 f(x,u,\bar w)=(x_0+u_0)^2+|x_1-1|+|x_2| \to \rm{inf},\\
 x_1=-x_0,\\
 x_2=x_1-u_1,\\
 x_0\le 1.\\
  \end{cases}
 \quad \quad \quad \quad \quad \quad  \quad \quad (P_2)
 $$
  Using Proposition \ref{Pro_Auxilary}, it is not difficult to see that $S(\bar w)=\{\bar z\}$ where $\bar z=(\bar x_0, \bar x_1,\bar x_2, \bar u_0, \bar u_1)=(-1,1,0,1,1)$.
Moreover, the linear operator $\Phi: \Bbb{R}^2\times \Bbb{R}^3 \times \Bbb{R}^2\rightarrow \Bbb{R}^2$,
  	 	$$\Phi(w,z)=Mz-Tw=\left(
  	 	\begin{array}{llll}
  	 	x_0+x_1+w_0\\
  	 	-x_1+x_2+u_1-w_1\\
  	 	 	\end{array}
  	 	\right),
  	 	$$
  	 	has closed range. The assumption ${\rm ker}\, T^* \subset {\rm ker}\,M^*$ is satisfied, because ${\rm ker}\, T^* = {\rm ker}\,M^*=\{(0,0)\}$. Hence, by Theorem \ref{Mainresult}, if $w^*=(w^*_0, w^*_1)\in \partial V(\bar w)$ then there exist $x_0^* \in N(\bar x_0; C)$, $\tilde{x}^*=(\tilde{x}_1^*, \tilde{x}_2^*)$, and $u^*=(u^*_0,u^*_1)\in N(\bar u; \Omega)$
such that \eqref{nondifferentiable} is satisfied.
It is clear that
  	 	\begin{align*}
  	 	&\partial h_2(\bar x_2)=\partial h_2(0)=[-1,1],\\
  	 	&  \partial_{x_0} h_0(\bar x_0, \bar u_0, \bar w_0)= \partial_{x_0}h_0(-1,1,0)=\{0\},\\
  	 		&  \partial_{x_1} h_1(\bar x_1, \bar u_1, \bar w_1)= \partial_{x_1}h_1(1,1,0)=[-1,1],\\
  	 			&  \partial_{u_0} h_0(\bar x_0, \bar u_0, \bar w_0)= \partial_{u_0}h_0(-1,1,0)=\{0\},\\
  	 	&  \partial_{u_1} h_1(\bar x_1, \bar u_1, \bar w_1)= \partial_{u_1}h_1(1,1,0)=\{0\},\\
  	 		&  \partial_{w_0} h_0(\bar x_0, \bar u_0, \bar w_0)= \partial_{w_0}h_0(-1,1,0)=\{0\},\\
  	 	&  \partial_{w_1} h_1(\bar x_1, \bar u_1, \bar w_1)= \partial_{w_1}h_1(1,1,0)=[-1,1].	
  	 	\end{align*}
 In addition, we have $A_0^*=-1,$ $ A^*_1=1,$ $B_0^*=0,$ $ B_1^*=-1$, $T_0^*=-1$, and $T_1^*=1.$ Therefore, \eqref{nondifferentiable} yields $\tilde{x}^*_2\in [-1,1]$, $ \tilde{x}^*_1\in [-2,2]$, $x^*_0\in [-2,2]$, $u^*_0=0, $ $u^*_1\in [-1,1]$. Combining these with the conditions $x_0^* \in N(\bar x_0; C)=\{0\}$, $u^*=(u^*_0,u^*_1)\in N(\bar u; \Omega)=\{(0,0)\}$, we obtain $\tilde{x}^*_2\in [-1,1]$, $ \tilde{x}^*_1\in [-2,2]$, $x^*_0=0$, $u^*_0=u^*_1=0$.
 Thus, the last $N$ inclusions of \eqref{nondifferentiable} imply
 $$
      {w}^*_0\in[-2,2], \,
    w^*_1 \in [-2,2].
 $$
 Therefore $\partial V(\bar w)\subset [-2,2]\times [-2,2].$ 
  	 \end{example}
  	 
 	\begin{acknowledgements}
  	The research of Duong Thi Viet An was supported by College of Sciences, Thai Nguyen University, Vietnam. The research of Nguyen Thi Toan was supported by the National Foundation for Science and Technology Development (Vietnam) under grant number 101.01-2015.04. The authors thank Prof.~Nguyen Dong Yen for useful discussions and the anonymous referees for valuable remarks.
  	\end{acknowledgements}


\begin{thebibliography}{99}
  	\bibitem{AnYao} An, D.T.V., Yao, J.-C: Further results on differential stability of convex optimization problems. J. Optim. Theory Appl. \textbf{170}, 28--42 (2016)
  			\bibitem{AnYen} An, D.T.V., Yen, N.D.:  Differential stability of convex optimization problems under inclusion
  		  		constraints. Appl. Anal. \textbf{94}, 108--128 (2015)
 \bibitem{Bartl_2008} Bartl D.: A short algebraic proof of the Farkas lemma. SIAM J. Optim.
 \textbf{19}, 234--239 (2008)
  \bibitem{Bertsekas} Bertsekas, D.P.: Dynamic Programming and Optimal Control Volume I. Athena Scientific. Belmont, Massachusetts (2005)
 \bibitem{Bonnans_Shapiro_2000} Bonnans, J.F., Shapiro, A.: Perturbation Analysis of Optimization Problems. Springer. New York. (2000) 
 \bibitem{Chieu_Kien_Toan} Chieu, N.H., Kien, B.T., Toan, N.T.:
    Further results on subgradients  of the value function to a parametric optimal control problem. J. Optim. Theory Appl. \textbf{168}, 785--801 (2016)
 \bibitem{Chieu_Yao} Chieu, N.H., Yao, J.-C.: Subgradients of the optimal value function in a parametric discrete
 optimal control problem. J. Ind. Manag. Optim. \textbf{6}, 401--410  (2010)
 
 
 \bibitem{Ioffe_Tihomirov_1979} Ioffe, A.D., Tihomirov, V.M.: Theory of Extremal Problems. North-Holland Publishing Company. North-Holland (1979) 
 
 \bibitem{Kien_Liou_Wong_Yao} Kien, B.T., Liou, Y.C., Wong, N.-C., Yao, J.-C.: Subgradients of value functions in parametric dynamic programming. European J. Oper. Res. \textbf{193}, 12--22 (2009)  
\bibitem{Mordukhovich_2006} Mordukhovich, B.S.: Variational Analysis and Generalized Differentiation. Volume I: Basic Theory, Volume II: Applications. Springer. Berlin (2006)	

\bibitem{MordukhovichNamYen2009} Mordukhovich, B.S., Nam, N.M., Yen, N.D.: Subgradients of marginal functions in parametric mathematical programming. Math. Program. Ser. B \textbf{ 116}, 369--396 (2009)	\bibitem{Mouss1} Moussaoui, M., Seeger, A.: Sensitivity analysis of optimal value functions of convex parametric programs with	possibly empty solution sets. SIAM J. Optim. \textbf{4},	659--675 (1994)

	\bibitem{TRk1} Rockafellar, R.T.: Hamilton-Jacobi theory and parametric
 	  		analysis in fully convex problems of optimal control. J. Global Optim. \textbf{248}, 419--431  (2004)	
 	\bibitem{Seeger} Seeger, A.: Subgradient of optimal-value function
   		in dynamic programming: the case of convex system  without optimal
   		paths. Math. Oper. Res. \textbf{21}, 555--575  (1996)	
	\bibitem{Toan_Thuy} Thuy, L.Q., Toan, N.T.: Subgradients  of the value function in a parametric  convex optimal control problem. J. Optim. Theory Appl. \textbf{170}, 43--64 (2016)
	\bibitem{Toan_2015} Toan, N.T: Mordukhovich Subgradients  of the value function in a parametric  optimal control problem. Taiwanese J. Math. \textbf{19}, 1051--1072  (2015) 	
	 		
     	\bibitem{Toan_Kien} Toan, N.T., Kien, B.T.: Subgradients  of the value function to a parametric  optimal control problem. Set-Valued Var. Anal. \textbf{18}, 183--203  (2010)	
     	  	
 	\bibitem{Toan_Yao} Toan, N.T., Yao, J.-C.: Mordukhovich subgradients of the value function to a
  	  		parametric discrete optimal control problem. J. Glob. Optim. \textbf{58}, 595--612  (2014)	
  	    	  	\bibitem{Tu} Tu, P.N.V.: Introductory Optimization Dynamics.
  	  	Springer-Verlag. Berlin  (1991)
  	\bibitem{BV1} Vinter, R.B.: Optimal Control. Birkh\"auser.
  	Boston  (2000)
  	
  	\end{thebibliography}
\end{document}